\documentclass[amssymb,10pt]{amsart}
\usepackage{latexsym}
\usepackage{amssymb}
\newtheorem{theorem}{Theorem}
\newtheorem{propo}{Proposition}[section]

\newtheorem{lemma}[propo]{Lemma}
\newtheorem{corol}[propo]{Corollary}
\newtheorem{corollary}[theorem]{Corollary}

\newtheorem{remark}[propo]{Remark}

\newcommand{\GL}{\operatorname{GL}}

\newcommand{\tr}{{\mathrm {tr}}}

\newcommand{\Spec}{{\mathrm {Spec}}}

\newcommand{\CC}{{\mathbb C}}

\newcommand{\QQ}{{\mathbb Q}}
\newcommand{\ZZ}{{\mathbb Z}}
\newcommand{\NN}{{\mathbb N}}

\newcommand{\FF}{{\mathbb F}}

\newcommand{\PSp}{\mathrm{PSp}}

\def\F{\mathbb F}
\def\N{\mathbb N}
\def\l{\lambda}

\begin{document}
\title [Zero-one generation laws for finite simple groups]
{Zero-one generation laws for finite simple groups}
\author{Robert M. Guralnick}
\address{Department of Mathematics, University of Southern California,
Los Angeles, CA 90089-2532, USA}
\email{guralnic@usc.edu}
 \author{Martin W. Liebeck}
\address{Department of Mathematics,
 Imperial College, London SW7 2BZ, UK}
\email{m.liebeck@imperial.ac.uk}
\author{Frank L\"ubeck}
\address{ Lehrstuhl D f\"ur Mathematik, Pontdriesch 14/16, 
52062 Aachen, Germany}
\email{ Frank.Luebeck@Math.RWTH-Aachen.De}
\author{Aner Shalev}
\address {Einstein Institute of Mathematics, The Hebrew University of 
Jerusalem, Jerusalem, 91904, Israel}
\email{shalev@math.huji.ac.il}


\begin{abstract}
Let $G$ be a simple algebraic group over the algebraic closure of $\FF_p$ ($p$ prime), and 
let $G(q)$ denote a corresponding finite group of Lie type over $\FF_q$, where $q$ is a power of $p$.
Let $X$ be an irreducible subvariety of $G^r$ for some $r\ge 2$. We prove a zero-one law for the probability that $G(q)$ is generated by a random $r$-tuple in $X(q) = X\cap G(q)^r$: the limit of this probability as $q$ increases (through values of $q$ for which $X$ is stable under the Frobenius morphism defining $G(q)$) is either 1 or 0. Indeed, to ensure that this limit is 1, one only needs $G(q)$ to be generated by an $r$-tuple in $X(q)$ for two sufficiently large values of $q$. 
We also prove a version of this result where the underlying characteristic is allowed to vary.

In our main application, we apply these results to the case where $r=2$ and the irreducible subvariety $X = C\times D$, a  product of two conjugacy classes of elements of finite order in $G$. This leads to new results on random $(2,3)$-generation of finite simple groups $G(q)$ of exceptional Lie type: provided $G(q)$ is not a Suzuki group, we show that the probability
that a random involution and a random element of order 3
generate $G(q)$ tends to $1$ as $q \rightarrow \infty$.
Combining this with previous results for classical groups, this shows that
finite simple groups (apart from Suzuki groups and 
$\PSp_4(q)$) are randomly $(2,3)$-generated.

Our tools include algebraic geometry, representation theory of algebraic groups, and character theory of finite groups of Lie type. 
\end{abstract}

\date{\today}



\thanks{The first author acknowledges the support of the NSF 
grant DMS-1600056. 
    The third author was supported by project B3 of SFB-TRR 195 `Symbolic
    Tools in Mathematics and their Application' of the German Research 
    Foundation (DFG).
The fourth author acknowledges the support of ISF grant 686/17 and the Vinik chair of mathematics which he holds. The authors also acknowledge the support of the National Science Foundation under Grant No. DMS-1440140 while they were in residence at the Mathematical Sciences Research Institute in Berkeley, California, during the Spring 2018 semester.
We also thank Michael Larsen for helpful discussions.}

\maketitle

\section{Introduction}

Let $p$ be a prime, let $k$ be the algebraic closure of the prime field $\FF_p$, and let $G$ be a simple, simply connected algebraic group over $k$. For a power $q = p^a$, denote by 
$G(q)$ a corresponding finite group of Lie type (twisted or untwisted). For a positive integer $r$, write $G^r$ for the direct product of $r$ copies of $G$. 
In this paper we prove a zero-one generation law for irreducible subvarieties $X$ of $G^r$. Let $X(q) = X\cap G(q)^r$, let $Q$ be the set of values of $q$ for which $X$
is stable under the Frobenius morphism defining $G(q)$, and   denote by $P_{G(q)}(X(q))$ the probability that an $r$-tuple chosen uniformly at random from $X(q)$ generates $G(q)$. Theorem \ref{newmain1} below shows that 
the limit ${\rm lim}_{q \in Q}P_{G(q)}(X(q))$ must be either 1 or 0. Moreover, it is 1 provided that $X(q)$ generates $G(q)$ for at least two values of $q$, one of which is greater than a specified constant. 

In order to study these limiting generation probabilities, one needs similar results for groups of fixed Lie type in varying characteristics. Theorem \ref{newmain2} below is such a result. Here $G$ is taken to be a Chevalley group scheme (defined over $\ZZ$) and $X$ to be an irreducible subscheme defined over a ring of the form ${\mathcal O}[\frac{1}{n}]$, where ${\mathcal O}$ is the ring of integers of an algebraic number field, and a similar zero-one law is obtained, providing finitely many characteristics are excluded. 


Let $T$ be a finite non-abelian simple group. A non-empty subset 
of $T$ is called normal if it is closed under conjugation by elements 
of $T$ (in other words, it is a union of conjugacy classes).  
Let $C,D$ be two normal subsets of $T$ .
Let $P_T(C,D)$  denote the probability that a random
pair of elements in $C \times D$ generates $T$ (that is, the proportion
of generating pairs among all pairs in $C \times D$).

If $C$ is the set of elements in $T$ of order $r$ and $D$ is the set of
elements of order $s$, then we define $P_{r,s}(T)=P_T(C,D)$.  
This has been studied extensively in recent years.  
Previous results include:

\begin{itemize}

\item[(1)] \cite{KL, lsprob} $P_T(T,T) \rightarrow 1$ for simple groups $T$, as 
$|T| \rightarrow \infty$;

\item[(2)] \cite{lish23}  $P_{2,3}(T) \rightarrow 1$ as 
$|T| \rightarrow \infty$ for classical groups $T$ not of type $\PSp_4$, 
and for alternating groups;

\item[(3)] \cite{lishrs}  Let $r,s$ be primes with $(r,s) \ne (2,2)$. Then
$P_{r,s}(T) \rightarrow 1$ as $|T| \rightarrow \infty$ for classical 
groups $T$ of large Lie rank (depending upon $(r,s)$), and for alternating groups;

\item[(4)]  \cite{BGK,GK}   In every simple group $T$ there exists a conjugacy class $C$  such that for any non-empty normal subset 
$D \subset T \setminus \{ 1 \}$, we have $P_T(C,D) \ge \frac{13}{42}$.
\end{itemize}

Most of the work in proving such results involves the analysis
of the simple groups of Lie type.  This typically splits into two cases  
--- the first when  the rank of the group goes to infinity and the second when the rank
is fixed and the field size $q$ goes to infinity.
The methods are usually quite different.  In the first case,
one can avoid exceptional groups.  

For example, by (3) above $P_{r,s}(T) \rightarrow 1$ as the rank of the
classical group $T$ goes to infinity. But if the rank is fixed, and
is not large enough (given $r,s$) this may fail: first of all, elements
of orders $r,s$ need not exist, since $r$ and $s$ may not divide $|T|$. 
But it may even fail with this divisibility assumption.  For example,
$P_{2,3}(\PSp_4(q)) \rightarrow \frac{1}{2}$ if $(q,6) = 1$ and 
$q \rightarrow \infty$, and $P_{2,3}(\PSp_4(q)) =0$ if $(q,6) \ne 1$ (see \cite{lish23}).

In the case of fixed rank, we shall use Theorems \ref{newmain1} and \ref{newmain2} 
to show that if $G(q)$ is a finite group of fixed Lie type over $\FF_q$, and $C,D$ are conjugacy classes of the
corresponding simple algebraic group, then $P_{G(q)}(C(q),\,D(q))$ is either $0$ or 
is at least  $1 - O(q^{-\frac{1}{2}})$ (so in particular goes to $1$ as 
$q \rightarrow \infty$) -- see Corollaries \ref{mainp}, \ref{mainvary}. Note that there are several very well-known zero-one laws in probability, such as those of Borel-Cantelli, Kolmogorov and Levy.

As another application, we  show in Corollary \ref{prob} that $P_{G(q)}(G(q),\,G(q)) \rightarrow 1$
as $q\to \infty$  -- that is, $G(q)$ is generated by a random pair of elements with probability tending to 1.  This was originally proved in \cite{KL} for $G(q)$ of classical type and in \cite{lsprob} for $G(q)$ of exceptional type, using completely different methods.

Now let $r,s$ be positive integers, and  consider elements of orders $r,s$ in $G(q)$. Since
there are only finitely many conjugacy classes of such elements in
the corresponding simple algebraic group, it follows from the above result that the limit points as $q \to \infty$ for 
$P_{r,s}(G(q))$ form a finite set of rational numbers between $0$ and $1$. Usually,
this set is expected to be $\{1\}$, but there are exceptions -- for example, as mentioned above, for $G(q) = \PSp_4(q)$ 
and $(r,s)=(2,3)$  it is $\{0,\frac{1}{2},1\}$.  

There is particular interest in the case where $(r,s) = (2,3)$, one reason being that the $(2,3)$-generated groups are precisely the images of the modular group $PSL_2(\ZZ)$. For alternating and classical groups $G$, the limiting behaviour of $P_{2,3}(G)$ was determined in \cite{lish23}. Here we shall use the results discussed above to complete the picture for exceptional groups in Theorem \ref{23excep}.

We now state our main results. 
As above, let $G$ be a fixed type of simply connected simple algebraic
group over $k$, the algebraic closure of $\FF_p$. For $q$ a power of $p$, 
let $F_q$  be a Frobenius  endomorphism of $G(k)$
such that the fixed point group $G^{F_q} = G(q)$ is a finite group of Lie type over $\FF_q$. Here $G(q)$ may be of twisted or untwisted type, but we adopt the assumption in the statements of the results below that as $q$ varies, 
\begin{center}
{\it either all $G(q)$ are of untwisted type, or all are of a fixed twisted type}.
\end{center}
For Suzuki and Ree groups $^2\!B_2(q)$, $^2\!G_2(q)$, $^2\!F_4(q)$, our notation for $q$ denotes the relevant odd power of 2 or 3, rather than its square root.

For a positive integer $r$, denote by $G^r$ the direct product of $r$ copies of $G$. 
If $Y$ is a subset of $G(q)^r$ for some $r\ge 2$, we write $P_{G(q)}(Y)$ for the probability that a random $r$-tuple in $Y$ generates $G(q)$. 
Also for an $r$-tuple $y=(y_1,\ldots ,y_r) \in  Y$, we write $\langle y\rangle$ to denote the subgroup $\langle y_1,\ldots ,y_r\rangle$. 

Our first two theorems  are rather general results about probabilistic generation by $r$-tuples in an  irreducible subvariety $X$ of $G^r$. In the first result, the characteristic 
 of the underlying field is fixed. In the second, the characteristic is allowed to vary, hence the requirement that $X$ is defined over a suitable ring in characteristic zero. 

In the statement below, we refer to an explicit positive integer $M$ (depending only on the rank of $G$) which is defined in Remark \ref{defM} in Section 2. Also, 
for a prime power $q$, we write $X(q) = X \cap G(q)^r$, and extend the action of the Frobenius morphism $F_q$ to $G^r$, acting in the same way on each factor.
In order to avoid some technical complications in the proof when $G(q)$ is of Suzuki or Ree type, we make the extra assumption that $X$ is $G$-invariant in this case 
(i.e. $(x_1,\ldots,x_r) \in X, g\in G \Rightarrow (x_1^g,\ldots ,x_r^g) \in X$).
\begin{theorem} \label{newmain1} Let $p$ be a prime,  let $k$ be the algebraic closure of the prime field $\FF_p$,  and let 
$G$ be a simple, simply connected algebraic group over $k$. Suppose $X$ is an irreducible subvariety of $G^r$ ($r\ge 2$), and let $Q$ be the set of powers $q=p^a$ such that $X$ is $F_q$-stable, where $F_q$ is the Frobenius morphism defining $G(q)$ as above. If $G(q)$ is of Suzuki or Ree type, assume further that $X$ is $G$-invariant. 
Then the following are equivalent:
\begin{itemize}
\item[{\rm (i)}] ${\rm lim}_{q\in Q,\,q\to \infty}P_{G(q)}(X(q)) = 1$;
\item[{\rm (ii)}] there exist $q_1,q_2 \in Q$, with $q_2>M$, and $r$-tuples $x_i \in X(q_i)$ such that $\langle x_i \rangle = G(q_i)$ for $i=1,2$.
\end{itemize}
\end{theorem}

\noindent \textbf{Remarks} (1) Of course if condition (ii) does not hold, then the limiting probability referred to in (i) is 0. Hence this theorem is a rather strong form of  zero-one law for random generation in a fixed characteristic. 

\noindent (2) As indicated above, in the statement of the theorem the groups $G(q)$ are taken to be either all of untwisted type, or all of a fixed twisted type. Thus for example when $G = E_6$, Theorem \ref{newmain1} comprises two assertions, one for the family $E_6(q)$ and the other for the family $^2\!E_6(q)$. To make it absolutely clear, for the latter family, the assertion is that the statements
\begin{itemize}
\item[{\rm (i)}] ${\rm lim}_{q\in Q,\,q\to \infty}P_{\,^2\!E_6(q)}(X(q)) = 1$;
\item[{\rm (ii)}] there exist $q_1,q_2 \in Q$, with $q_2>M$, and $r$-tuples $x_i \in X(q_i)$ such that $\langle x_i \rangle = \,^2\!E_6(q_i)$ for $i=1,2$
\end{itemize}
are equivalent. This observation applies to all the results stated below.

\noindent (3) The set $Q$ defined in the statement of Theorem \ref{newmain1} can be described as follows. Let $q_0$ be minimal such that  
$X$ is $F_{q_0}$-stable. If the groups $G(q)$ are untwisted (so that the $F_q$ are field endomorphisms), 
then $Q = \{q_0^a : a\in \NN\}$, the set of powers of $q_0$. If $G(q)$ is of Suzuki or Ree type, $Q$ consists of all odd powers of $q_0$. And for the other twisted types, $Q$ consists of either all powers of $q_0$, or all powers $q_0^a$ with $a$ coprime to 2 (for types $^2\!A_n,\,^2\!D_n,\,^2\!E_6$) or coprime to 3 (for type $^3\!D_4$); the first possibility (all powers) occurs 
if $X$ is fixed by a graph automorphism of $G$, the second if not.  

\noindent (4) If the irreducibility condition on $X$ is dropped, then application of the theorem to the irreducible components of highest dimension shows that the probability defined in (i) is a rational number.

\vspace{2mm}
In the next result the characteristic is allowed to vary, and in order to state it  we require some notation. Let $R$ be a ring of the form ${\mathcal O}[\frac{1}{n}]$, where $n$ is a positive integer, and ${\mathcal O}$ is the ring of integers of an algebraic number field. 
Suppose $G$ is a simply connected Chevalley group scheme (defined over $\ZZ$), and let $X$ be an irreducible subscheme of $G^r$ defined over $R$, for some $r\ge 2$. 
For each prime $p$, we have the simply connected simple algebraic group $G(\bar \FF_p)$, and Frobenius endomorphisms $F_q$ ($q=p^a$) with fixed point groups $G(q)$ as above (again assuming either all $G(q)$ are of untwisted type, or all are of a fixed twisted type). As before, extend the action of $F_q$ to $G(\bar \FF_p)^r$, acting in the same way on each factor.
For $p$ coprime to $n$, let $q_i(p)$ ($1\le i\le k_p$) denote the sizes of the residue fields of $R$ in characteristic $p$, and for a finite set $S$ of primes such that $S$ contains all the prime divisors of $n$, define 
\[
Q_R(S) = \bigcup_{p\not \in S}\bigcup_i\{q_i(p)^a : X(\bar \FF_p) \hbox{ is }F_{q_i(p)^a}\hbox{-stable}\}.
\]
For $q\in Q_R(S)$, set $X(q) = G(q) \cap X(\bar \FF_p)$. 

\begin{theorem}\label{newmain2}
Let $G$ and $X$ be as above. Then there is a finite set $S$ of primes (containing the prime divisors of $n$) such that  the following are equivalent:
\begin{itemize}
\item[{\rm (i)}] ${\rm lim}_{q\in Q_R(S), q\to \infty} P_{G(q)}(X(q)) = 1$;
\item[{\rm (ii)}] there exist a prime $p \not \in S$,  two powers $q,q'$ of $p$ lying in $Q_R(S)$, with $q'>M$, and $r$-tuples $x \in X(q)$, $x'\in X(q')$, such that $\langle x \rangle = G(q)$, 
$\langle x' \rangle = G(q')$.
\end{itemize}
\end{theorem}

\noindent \textbf{Remarks }(1) If in Theorems \ref{newmain1} and \ref{newmain2} we also assume that $\dim X > \dim G$, then in both results, condition (ii) can be replaced by the existence of just one $r$-tuple $x' \in X(q')$  (rather than two) such that $\langle x' \rangle = G(q')$ and $q'>M$. This remark is justified at the end of the proof of Theorem \ref{newmain2} in Section \ref{th2pf}.

(2) In our main applications of Theorem \ref{newmain2}, namely Corollaries \ref{prob} and \ref{mainvary} below, $X$ is defined over a ring $R$ whose field of fractions is a cyclotomic extension of $\QQ$, and so all the residue fields of $R$ in a given characteristic have the same size (i.e. $k_p=1$ for all $p$ in the above notation).

\vspace{2mm}
There are many interesting subvarieties of $G^r$ to which we can attempt to apply Theorems \ref{newmain1} and \ref{newmain2} -- for example $G^r$ itself, or the Cartesian product $C_1\times \cdots \times C_r$ of conjugacy classes $C_i$, or a representation variety ${\rm Hom}(\Gamma, G)$ of a finitely presented group $\Gamma$ on $r$ generators. We shall apply the above results to several such subvarieties.

First, we take $r=2$ and $X = G^2$, which is of course irreducible. Applying Theorems \ref{newmain1} and \ref{newmain2}, we can quickly deduce the following consequence, giving random generation of simple groups of fixed Lie type by pairs of elements. This was first proved for classical types in \cite{KL} and for exceptional types in \cite{lsprob}, using completely different methods. 

\begin{corollary}\label{prob} Let $G$ be a fixed simply connected, simple Lie type. Then 
\[
{\rm lim}_{q\to \infty}P_{G(q)}(G(q),G(q)) = 1.
\]
\end{corollary}

More generally, for $X = C\times G$, where $C$ is an irreducible  subvariety of $G$, we prove the following; we state it in fixed characteristic, but as for Theorem \ref{newmain2} there is a version for varying characteristics, provided we asume that $C$ is a subscheme defined over a suitable ring. We write $C(q) = C\cap G(q)$.

\begin{corollary}\label{cgcor} Let $G$ be a simple, simply connected algebraic group over $\bar \FF_p$, let $C$ be an irreducible subvariety of $G$, and let $Q$ be the set of powers $q$ of $p$ such that $C$ is $F_q$-stable.  If $G(q)$ is of Suzuki or Ree type, assume further that $C$ is $G$-invariant. Then 
${\rm lim}_{q\in Q,\,q \to \infty}P_{G(q)}(C(q),G(q)) = 1$.
\end{corollary}

In the next  two corollaries, we take $r=2$ and $X = C\times D$, the Cartesian product of two conjugacy classes. In the first,                                                                                                                                                                                                                                                                                                                                                                                                                                                                                                                                                                                                                                                                                                                                                                                                                                                                                                                                                                                                                                                                                                                                                                                                                                                                                                                                                                                                                                                                                                                                                                                                                                                                                                                                                                                                                                                                                                                                                                                                                                                                                                                                                                                                                                                                                                                                                                                                                                                                                                                                                                                                                                                                                                                                                                                                                                                                                                                                                                                                                                                                                                                                                                                                                                                                                                                                                                                                                                                                                                                                                                                                                                                                                                                        the characteristic is fixed, and the classes can contain elements of arbitrary (finite) order; while in the second,the characteristic of the underlying field is allowed to vary, but the classes must consist of semisimple elements. Write $C(q) = C\cap G(q)$, $D(q) = D\cap G(q)$.

\begin{corollary} \label{mainp} Let $p$ be a prime and $k = \bar \FF_p$,  and let  
$G$ a simple, simply connected algebraic group over $k$.
Let $C$ and $D$ be  conjugacy classes in $G$ consisting of elements
of finite order, and denote by $Q$ the set of powers $q=p^a$ such that $C(q) \ne \emptyset$ and $D(q) \ne \emptyset$. Then the following are equivalent:
\begin{itemize}
\item[{\rm (i)}] ${\rm lim}_{q\in Q, q\to \infty}P_{G(q)}(C(q),D(q)) = 1$;
\item[{\rm (ii)}] there exists $q > M$ and $(c,d) \in C(q)\times D(q)$, such that $\langle c,d\rangle = G(q)$.
\end{itemize}
\end{corollary}

Notice that in (ii), we only require the condition for one value of $q$ (rather then the two required in (ii) of Theorem \ref{newmain1}). This is justified by Remark (1) after Theorem \ref{newmain2}, together with Lemma \ref{scottlemma} which shows that condition (ii) implies that $\dim C+\dim D > \dim G$.

To state the second corollary, we need some notation for semisimple classes. 
For a positive integer $r$, the conjugacy classes of semisimple elements of order $r$ in $G(k)$ are parametrized by the orbits of the Weyl group $W(G)$ on the set of elements of order $r$ in a maximal torus. This parametrization is independent of the characteristic $p$ (coprime to $r$). For such an orbit $\Delta_r$ of the Weyl group, write $C(\Delta_r,k)$ for the corresponding semisimple class of elements of order $r$ in $G(k)$.

\begin{corollary}\label{mainvary}
Let $G$ be a fixed type of simple, simply connected algebraic group. Let $r,s$ be positive integers, and $\Delta_r,\Delta_s$ orbits of the Weyl group, as above. For any algebraically closed field $k$ of characteristic coprime to $rs$, let $C = 
C(\Delta_r,k)$ and $D = C(\Delta_s,k)$ be the corresponding classes of semisimple elements of orders $r,s$ in $G(k)$. Denote by $Q$ the set of prime powers $q$ (in varying such characteristics) such that $C(q) \ne \emptyset$ and $D(q) \ne \emptyset$. Then  the following are equivalent:
\begin{itemize}
\item[{\rm (i)}] ${\rm lim}_{q\in Q, q\to \infty}P_{G(q)}(C(q),D(q)) = 1$;
\item[{\rm (ii)}] for any prime $p \nmid rs$, there is a power $q=p^a>M$ and a pair $(c,d) \in C(q)\times D(q)$, such that $\langle c,d\rangle = G(q)$.
\end{itemize}
\end{corollary}

In fact, condition (ii) is only required to hold in finitely many characteristics in order to imply (i).




We now turn to some applications. For positive integers $r,s$ 
let  $V_{r,s}$ be the set of values $P_{r,s}(T)$
as $T$ ranges over all finite simple groups. Let $L_{r,s}$ denote
the set of limit points of $V_{r,s}$.  The following result is an easy consequence
of Corollaries \ref{mainp} and \ref{mainvary}, together with \cite{lishrs}.

\begin{corollary}  \label{rscor} Fix primes $r,s$.
Then  $L_{r,s}$  is a finite set of rational numbers.
\end{corollary}

We conjecture that for $\{ r,s \} \not\subseteq \{ 2,3 \}$, we have  
$L_{r,s}=\{ 0, 1\}$.  More specifically, we conjecture that if $G_i$ is a sequence
of finite simple groups 
such that $|G_i| \rightarrow \infty$ with $r$ and $s$ both dividing $|G_i|$,
then $P_{r,s}(G_i) \rightarrow 1$.  When the ranks of all the groups $G_i$ are greater then some function $f(r,s)$,
this has been proved in \cite{lishrs}; and it has been proved for all ranks when the $G_i$ are of type $A_r$ or $^2\!A_r$, 
in \cite[Cor. 1.3]{Ger}.

In the next result we apply Corollaries \ref{mainp} and \ref{mainvary}  to the study of probabilistic $(2,3)$-generation, completing
the work in \cite{lish23}.  The proof makes essential use of the result of L\"ubeck and Malle
\cite{LM} that, apart from Suzuki groups, all exceptional
groups of Lie type are $(2,3)$-generated. 

\begin{theorem} \label{23excep}  Let $G_i$ be a sequence of 
finite simple exceptional groups of  Lie type with $|G_i| \rightarrow \infty$.
Assume that none of the $G_i$ are Suzuki groups. 
Then $P_{2,3}(G_i) \rightarrow 1$ as $i \rightarrow \infty$. 
\end{theorem}

Note that, while in \cite{lish23}, $(2,3)$-generation of classical
groups is deduced from random $(2,3)$-generation established for
these groups, the
deduction here is in the reverse direction: we need to know that
exceptional groups are $(2,3)$-generated in order to prove
that they are randomly $(2,3)$-generated.

Combining  Theorem \ref{23excep}  with \cite{lish23} (for alternating
and classical groups), we obtain the following.

\begin{corollary} \label{23cor}  For $T$ simple, as $|T| \to \infty$, 
\[
P_{2,3}(T) \to \left\{ \begin{array}{l}
0,\;T = \,^2\!B_2(2^a),\,PSp_4(2^a),\,PSp_4(3^a) \\
\frac{1}{2},\,T = PSp_4(p^a),\,p\ne 2,3, p \hbox{ prime} \\
1,\,\hbox{ otherwise}.
\end{array}
\right.
\]
\end{corollary}

These random $(2,3)$-generation results have interesting applications
to residual properties of the modular group (see \cite{lishres}).

The paper comprises three further sections. Section 2 contains various preliminary results on modules and trace maps that are needed for our proofs of the main results. 
In Section 3, we prove Theorems \ref{newmain1}, \ref{newmain2} and Corollaries \ref{prob}--\ref{mainvary}. The final Section 4 contains proofs of our results on $(2,3)$-generation.

\section{Preliminary results on modules and traces}

\subsection{Modules}

We begin with an elementary lemma.

\begin{lemma} \label{elem} Let $k$ be an algebraically closed field and $n$
a positive integer.  The set of pairs $(A,B)$ in 
$\GL_n(k) \times \GL_n(k)$ that
generate an irreducible subgroup of $\GL_n(k)$ is a Zariski dense open subset
of $\GL_n(k) \times \GL_n(k)$.
\end{lemma}

\begin{proof}  First observe that by the Artin-Wedderburn theorem, $\langle A,B\rangle$ is an irreducible subgroup of $\GL_n(k)$ if and only if the algebra generated by $A$ and $B$ is the full matrix algebra $M_n(k)$. If we write $S^i$ for the set of words in $A,B$ of length at most $i$, and $kS^i$ for the linear span of $S^i$, then 
\[
k = kS^0 \subseteq kS^1 \subseteq \cdots kS^i \subseteq kS^{i+1} \subseteq \cdots .
\]
This chain must stabilize at some point -- that is, $kS^i=kS^{i+1}=\cdots$ for some $i$, and clearly $i\le n^2$.
Thus $\langle A,B\rangle$ is irreducible if and only if the linear span of the set of words in $A,B$ of length at most $n^2$ is equal to $M_n(k)$. This is an open condition, so the set of pairs generating an irreducible subgroup is open. Finally, it is non-empty, since $M_n(k)$ can be generated by 2 elements -- for example, it is generated by a diagonal matrix with distinct eigenvalues, together with a matrix of the form $\lambda I+J$, where $J$ is the all 1's matrix and $\lambda$ is a scalar.
\end{proof}


In the next result, recall that for a simple algebraic group $G$ over a field $k$ of characteristic $p$, and a Frobenius endomorphism $F_q$ ($q$ a power of $p$), we write $G^{F_q}=G(q)$, a group of Lie type over $\FF_q$ which can be of untwisted or twisted type.

\begin{lemma}\label{goodmodules}  Let $G$ be a fixed simply connected, simple Lie type of rank $r$, and let $k$ be an algebraically closed field of positive characteristic $p$. There is a finite set ${\mathcal N}_G$ of integers, and an absolute constant $K$ such that the following hold.
\begin{itemize}
\item[{\rm (i)}] There exists a finite collection ${\mathcal S}$ of finite-dimensional irreducible $G(k)$-modules 
such that the set
of proper closed subgroups of $G(k)$ that act irreducibly on every member of ${\mathcal S}$ is conjugate to a group in  
$\{G(q) : q \not \in {\mathcal N}_G\}$ (where the groups $G(q)$ can be of untwisted or twisted type). Moreover, ${\rm max}\{n : n \in {\mathcal N}_G\} < f(r)$, a function of the rank only.   
\item[{\rm (ii)}] For $p>K$, all the modules in ${\mathcal S}$ are restricted modules, and the set consisting of their highest weights is independent of $p$.
\end{itemize}
\end{lemma}

\begin{proof} 
This follows from results in \cite{gurtiep} and \cite{lieseiadj}, as follows.
Define $K$ to be the maximum order of a group that appears either in \cite[Table 1.1]{lieseiadj} or collection ${\mathcal E}_1$ of \cite[Theorem 2.9]{gurtiep}; so in fact, $K$ is equal to the order of the Monster sporadic group. 

For $p\le K$ we take ${\mathcal S}$ to consist of the single irreducible module provided by \cite[Theorem 11.7]{gurtiep}.
That result asserts that any proper closed subgroup that is irreducible on this module must be a subgroup of the form $G(q)$.

 For $p>K$ and $G = Cl(V)$ of classical type with $\dim V \ge 5$, we take ${\mathcal S}$ to consist of the set of modules given in \cite[Theorem 2.9]{gurtiep} -- that is, the composition factors of $V\otimes V^*$, $V^{\otimes 4}$ and $S^3(V)$ -- together with the irreducible of highest weight $K\omega_1$. Theorem 2.9 of \cite{gurtiep} states that any proper closed subgroup that is irreducible on all these modules is either a subgroup $G(q)$, or a member of the collection 
${\mathcal E}_1$ mentioned above; but members of this collection cannot be irreducible on the last module 
$V(K\omega_1)$, since this has dimension larger than the order of any of the groups in the collection, 
by the definition of $K$. 
For $G = Cl(V)$ with $\dim V\le 4$ (i.e. $G$ of type $A_1,A_2,A_3$ or $C_2$, still with $p>K$), it is straightforward to define a suitable collection ${\mathcal S}$, and we leave this to the reader. 

Finally, for $p>K$ and $G$ of exceptional type, we take ${\mathcal S}$ to consist of the adjoint module together with a further restricted module of dimension greater than $K$ (for example $V(K\omega_1)$ again), and the conclusion follows from \cite[Theorem 1]{lieseiadj}. 
\end{proof}

\begin{remark}\label{defM}{\rm  Let ${\mathcal N}_G$ be as defined in the lemma, and let $M = M_G = {\rm max}\{n : n \in {\mathcal N}_G\}$ (the maximum taken over all characteristics). Then for $q>M$, the group $G(q)$ acts irreducibly on all the modules in the collection ${\mathcal S}$.}
\end{remark}

Continue to let $G$ be a simply connected Lie type, and define ${\mathcal S}$, ${\mathcal N}_G$, $K$ as in the previous result.
 For a prime $p$, and $k = \bar \F_p$, and $r\ge 2$, define 
\begin{equation}\label{wdef}
W_r'(k) = \{x \in G(k)^r \;|\; {\langle x\rangle} \hbox{ acts reducibly on some module in }{\mathcal S}\}.
\end{equation}
Let $W_r(k)$ denote the complement of $W_r'(k)$ in $G(k)^r$.

We also need to define similar subsets in varying characteristics. To do this, now take 
 $G$ to be a Chevalley group scheme, and denote by $s$ the product of all the primes less than $K$.
Then Lemma \ref{goodmodules}(ii) shows that there is a subscheme $W_r'$ of $G^r$ defined over $\ZZ[\frac{1}{s}]$, such that for 
each $k = \bar \FF_p$ ($p>K$), $W_r'(k)$ is as defined in (\ref{wdef}). Let $W_r$ be the complement of $W_r'$.



\begin{lemma} \label{n large} Let $p$ be a prime, and $k = \bar \FF_p$.
\begin{itemize}
\item[{\rm (i)}]  Then $W_r(k) = \{x \in G(k)^r \;|\; G(q)^g \le \langle x\rangle \hbox{ for some }g\in G(k),\,q \not \in {\mathcal N}_G\}$.
\item[{\rm (ii)}]  $W_r(k)$ is a dense open subset of $G(k)^r$.
 \item[{\rm (iii)}] $W_r(k)$ is defined over $\F_p$.
\end{itemize}
\end{lemma}

\begin{proof} 
Part (i) follows from Lemma \ref{goodmodules}. 
For (ii), observe that  $W_r(k)$ is open by Lemma \ref{elem}, and it is non-empty since any group $G(q)$ can be generated by 2 elements (by \cite{St}).
Finally, (iii) holds since in fixed characteristic $p$, irreducible $G$-modules are defined over $\F_p$.
\end{proof}

\begin{lemma} \label{scottlemma} Let $k = \bar \FF_p$, and suppose that  $C,D$ are conjugacy classes of $G(k)$
such that  $(C \times  D) \cap W_2(k)$ is non-empty.
Then 
$\dim C + \dim D  > \dim G$. 
\end{lemma}

\begin{proof}   Let $A = {\rm Lie}(G)$, the Lie algebra of $G=G(k)$.  By assumption, we can choose $(x,y) \in C\times D$ such
that $\langle x, y \rangle$ contains $G(q)$ for some $q>M$, hence has the same set of invariant subspaces on $A$  as $G$ does.
Let $r = {\rm rank}(G)$.

Set $z=xy$.  Now applying Scott's Lemma \cite{scott}, we have
$$ 
\dim [x,A] + \dim [y,A] + \dim [z,A]  \ge \dim A + \dim [G,A] - \dim A^G.
$$
For $g \in G(k)$ we have $\dim [g,A] = \dim A  - \dim A^g \le \dim G - \dim C_G(g)$.
The last term
is the dimension of the conjugacy class of $g$, and so we have
$\dim [x,A] \le \dim C$ and $\dim [y,A] \le \dim D$.
Also, the dimension of the conjugacy class of $z$ is at most 
$\dim G - r$, whence $\dim [z,A] \le \dim G - r$. Combining
the above inequalities we obtain the inequality
\begin{equation}\label{ineqq}
\dim C + \dim D \ge   \dim [G,A] - \dim A^G + r.
\end{equation}
 
We claim that $A=[G,A]$. To see this, observe that $[G,A]$ is an ideal of $A = {\rm Lie}(G)$ and $G$ acts trivially on the quotient. Since $G$ is simply connected, 
inspection of \cite[Table 1]{hog} shows that there are no such proper ideals of $A$, which proves the claim. 
It follows that
$\dim [G,A] - \dim A^G = \dim G - \delta$, where
$\delta$ is at most the number of trivial composition factors
of $G$ on $A$.  This number is at most 1 unless $p=2$ and $G=B_r,C_r$ or $D_r$, all with $r$ even, in which case it is 2  (see \cite[Prop. 1.10]{lstams}). 
Now (\ref{ineqq}) gives $\dim C+\dim D \ge \dim G+r-\delta$, and the conclusion follows 
unless either  $G = A_1$ or $(G,p)=(C_2,2)$.  

Suppose $G=A_1$. Then 
every nontrivial conjugacy class has dimension at least $2$, whence $\dim C+\dim D > \dim G$, as required. 

Finally,  consider
$(G,p)=(C_2,2)$. Here $G$ has two conjugacy classes of dimension 4, namely the classes of long and short root elements, and the other classes have dimension at least 6. Assume that $\dim C + \dim D \le \dim G = 10$. Then adjusting $C,D$ by a graph automorphism of $G$ if necessary, we may take one of the classes, say $C$, to contain  a long root element $x = u_{\alpha}$. Then $x$ has fixed point space of dimension 3 on the natural 4-dimensional module $V$ for $G = Sp_4$. Now $\dim D \le 6$, and it is easy to see that an element $y$ in such a class must have an eigenspace on $V$ of dimension at least 2. This eigenspace intersects the fixed space of $x$ nontrivially, so $x,y$ cannot generate an irreducible subgroup of $G$, which is a contradiction. Hence $\dim C + \dim D > \dim G$, completing the proof. 
\end{proof}

If $k = \bar \FF_p$ and $C,D$ are conjugacy classes of $G=G(k)$ such that $\dim C + \dim D  \le \dim G$, the next lemma shows there are severe limitations on the possiblilities for the subgroup $\langle c,d\rangle$, where $(c,d) \in X = C\times D$.

 \begin{lemma}\label{xleg}
Let $X$ be an irreducible subvariety of $G^r$ ($r\ge 2$), and suppose that $X$ is invariant under conjugation by $G$.
\begin{itemize}
\item[{\rm (i)}] If $\dim X < \dim G$, then every $r$-tuple in $X$ generates a subgroup of a proper parabolic subgroup.
\item[{\rm (ii)}] If $\dim X=\dim G$, then one of the following holds:
\begin{itemize}
\item[{\rm (a)}] $G$ has a dense open orbit on $X$, and every $r$-tuple in the orbit generates a conjugate of a fixed finite group $H$ such that $C_G(H)$ is finite; the other $r$-tuples in $X$ generate a subgroup of order at most $|H|$ of a parabolic subgroup;
\item[{\rm (b)}] every $r$-tuple in $X$ generates a subgroup of a parabolic subgroup.
\end{itemize}
\end{itemize}
\end{lemma}

\begin{proof}
(i) Suppose $\dim X < \dim G$. Then any $x \in X$ has a positive-dimensional centralizer in $G$, hence $\langle x \rangle$ is contained in a parabolic subgroup. 

(ii) Now suppose $\dim X=\dim G$, and let $x \in X$. If the orbit $x^G$ has dimension less than $\dim G$, then $\langle x \rangle$ is contained in a parabolic subgroup, as in (i). Otherwise, $x^G$ is a dense orbit, and the centralizer $C_G(x)$ is a finite group. Part (ii) follows.
\end{proof}

\subsection{Traces}

\begin{lemma}\label{traces}
  Let $p$ be a prime and  $k = \bar \FF_p$. Let 
$S$ be a finite subset of $GL_d(k)$ that generates an irreducible
subgroup, and for $m\ge 1$ let $S^m$ denote the set of all words in
$S$ of length at most $m$.   Let $E$ be the subfield of $k$
generated by $\{\tr(T) \,:\, T \in S^{2d^2}\}$.   Then the $\FF_p$-algebra generated
by $S$ is $GL_d(k)$-conjugate to $M_d(E)$, and 
$\langle S \rangle$ is conjugate to a subgroup of $GL_d(E)$.  
\end{lemma}

\begin{proof}  By the Artin-Wedderburn theorem,  the $k$-algebra
generated by $S$ is $M_d(k)$.    Let $Y =\{Y_1, \ldots, Y_{d^2}\} \subset S^{d^2}$  be 
a basis for $M_d(F)$.   
 Note that the map $\phi : M_d(k) \rightarrow k^{d^2}$ given
by $A \rightarrow  (\tr(AY_1), \ldots, \tr(AY_{d^2}))$ is a linear bijection, 
since the $Y_i$ form a basis and the trace form is non-degenerate.

If $B \in Y$, then $BY \subset S^{2d^2}$ and so $\tr(BY_i) \in E$.  
Since the above map $\phi$ is a bijection, it follows that $BY_i$ is in the $E$-span of $Y$.   Thus, the $\FF_p$-algebra generated by
$S$ has cardinality $q^{d^2}$ where $q=|E|$.  Since it acts absolutely irreducibly
on $k^{d}$, it is isomorphic to $M_d(E)$. It follows also that $\langle S \rangle$ is isomorphic to a subgroup of $GL_d(E)$.  
\end{proof}


\begin{corol}\label{orderN}   Let $k$ be the algebraic closure of $\FF_p$, let $G$ be a simple
algebraic group over $k$, and let $V$ be a nontrivial irreducible restricted $d$-dimensional
rational $kG$-module.  Let $X$ be an irreducible subvariety of $G^r$ ($r\ge 2)$, and assume
that some $x \in X$ generates a subgroup acting irreducibly on $V$.  Suppose that for every word $w$ in the free
group on $r$ generators of length at most $2d^2$, the morphism $\tr(w(y))$ is constant
for $y$ in an nonempty open subset of $X$. Then the following hold:
\begin{itemize}
\item[{\rm (i)}] there is a constant $N$ such that the group generated by $y$ for all $y \in X$
has order at most $N$;
\item[{\rm (ii)}] $\dim X \le \dim G$.
\end{itemize}
\end{corol}

\begin{proof} (i)  The condition that the group generated by $y$ has order at most $N$
is a closed condition, so it suffices to prove the conclusion for $y$ in the nonempty open subset 
of $X$ consisting of $r$-tuples generating an irreducible subgroup.   By the previous lemma,
it follows that the group generated by any such $y$ is conjugate to a subgroup of $\GL_d(q)$,
where $\FF_q$ is the field generated by all traces of words of length at most $2d^2$
(the hypotheses imply this is a finite field). Part (i) follows, taking $N = |\GL_d(q)|$.

(ii) Let $H = \GL_d(q)$ be as above, and let $\phi : G\times H^r \to G^r$ be the morphism sending $(g,h_1,\ldots ,h_r) \to (h_1^g,\ldots ,h_r^g)$ ($g\in G, h_i \in H$).
Then ${\rm Im}(\phi)$ contains a nonempty open subset of $X$. Hence $\dim X \le \dim G$.
\end{proof}

In the next result, $R$ denotes a ring of the form ${\mathcal O}[\frac{1}{n}]$, where $n$ is a positive integer, and ${\mathcal O}$ is the ring of integers of an algebraic number field. As before,  for a prime $p$ coprime to $n$, we let $q_i(p)$ ($1\le i\le k_p$) denote the sizes of the residue fields of $R$ in characteristic $p$. Also define
\[
Q_R = \bigcup_{p\nmid n}\bigcup_i\{q_i(p)^a : a \in \N\}.
\]

\begin{lemma}\label{sch}  Suppose that $V$ is an irreducible reduced scheme of finite type of dimension $s$ defined over the ring $R={\mathcal O}[\frac{1}{n}]$.
Let $f$ be a non-constant regular function on $V$ which is defined over $R$. 
There exists a constant $c$ such that for any $q \in Q_R$, the set of $\FF_{q}$-points in any fiber of $f$ has
size at most $c q^{s-1}$.
\end{lemma}

\begin{proof}    This is well-known -- see for example \cite[Lemma 2.2]{LarS}.   
For completeness, here is a sketch of an elementary proof.   It suffices to assume that $V$ is affine.
 Let $A$  be an integral domain finitely generated over $R$  with fraction field $K$.
 Assume that $K$ has transcendence degree $s$ over $R$, and that  
 the result is known for varieties of dimension less than $s$. 
 
   Let $f_1\ldots.,f_s$ be a transcendence base of $K$ contained in $A$.
     Let $g$  be any element in $A$ not algebraic over  $R$.  Then, reordering if
     necessary,   $g,f_1 \ldots, ,f_{s-1}$  is a transcendence base (by the exchange lemma).  
     By generic freeness (see \cite[Thm. 4.4]{eis}), 
for some $b \in  B = R[g,f_1,\ldots, f_{s-1}]$, 
     we have that $A[1/b]$ is free over $B[1/b]$.   Let $ a_1,\ldots ,a_m$ be a basis.
     
       Every point of $\Spec A$ is either a point of  $\Spec A/bA$ or a point of $\Spec A[1/b]$.  The former case is covered by the induction hypothesis, so we consider the latter.  Consider any finite field  $\FF_q$ ($q \in Q_R$)  and any homomorphism $A[1/b] \rightarrow  \FF_q$.  If we fix first the image of $g$  (i.e., which fiber we are on), and then the images of $f_1,\ldots ,f_{s-1}$, then there are at most $m$ possibilities for the homomorphism.  Thus, there are at most  $m q^{s-1}$ points over $\FF_q$ in any $g$-fiber of $\Spec A[1/b]$.
\end{proof}

\begin{corol}\label{tra}  Let $V$ and $f$ be as in Lemma $\ref{sch}$. 
Then there exists a constant $c'$ such that  for any $q=q_p^a \in Q_R$, 
\[
|v \in V(q) \,:\,  \FF_{p}[f(v)] \ne \FF_{q}| <c'q^{(s-1)/2}.
\]
\end{corol}

\begin{proof}    For any finite field $F = \FF_{p^f}$ ($p$ prime), the number of  elements of $F$ that lie in a proper subfield
is at most  $S:=\sum_{\ell}  p^{f/\ell}$ where the sum is over all prime divisors $\ell$ of $f$.  The number of primes
dividing $f$ is at most $\log_2(f)$ and so 
$$
S \le p^{f/2}  + \log\log(p^f) p^{f/3} < 2 p^{f/2}.
$$
Now apply Lemma \ref{sch} to obtain the conclusion.
\end{proof} 

We shall also need a version of this result for a fixed characteristic:

\begin{lemma}\label{trap} Let $p$ be a prime and $k = \bar \FF_p$. Suppose that $V$ is an irreducible $k$-variety of dimension $s$, defined over $\FF_{q_0}$ ($q_0=p^a$),  
and $f:V\to k$ is a non-constant morphism defined over $\FF_{q_0}$. 
Then there exists a constant $c'$ such that  for any power $q$ of $q_0$, 
\[
|v \in V(q) \,:\,  \FF_{p}[f(v)] \ne \FF_{q}| <c'q^{(s-1)/2}.
\]
\end{lemma}

\begin{proof}
This is proved just as in Lemma \ref{sch} and Corollary \ref{tra}, where in the proof of Lemma \ref{sch} we  take $A$ to be an integral domain over $k$.
\end{proof}

\section{Proof of Theorems \ref{newmain1}, \ref{newmain2} and corollaries}

In this section we prove Theorems \ref{newmain1} and \ref{newmain2} and deduce Corollaries \ref{prob}--\ref{rscor}. 

\subsection{Proof of Theorem \ref{newmain1}}

Let $p, k, G, X$ and $Q$ be as in the statement of Theorem \ref{newmain1}. Clearly condition (i) implies (ii). So now assume condition (ii) holds -- that is, there exist $r$-tuples $x_i \in X(q_i)$ ($i=1,2$) such that $\langle x_i \rangle = G(q_i)$ for $i=1,2$, and $q_2 > M$, where $M$ is as defined in Remark \ref{defM}. 

Assume first that $G(q)$ is not a Suzuki or Ree group. We define an irreducible $kG$-module $V$ as follows. If $G(q)$ is of untwisted type,  let 
$V = V_G(\l_1)$, the irreducible module for the first fundamental dominant weight for $G$; 
and if $G(q)$ is twisted, let $V$ be a composition factor of the adjoint module of largest dimension. 
Let $\chi: G\to k$ be the character of this module. Then for a power $p^f$ of $p$, the restriction $V\downarrow G(p^f)$ is realised over $\F_{p^{f}}$. Also, it is not realised over a proper subfield of $\F_{p^{f}}$, as the highest weight is not fixed by a field automorphism of $G(p^f)$, so the trace values $\chi(G(p^f))$ generate $\F_{p^f}$.

Now consider the irreducible subvariety $X \subseteq G^r$. Suppose the morphism $\chi\circ w$ is constant on $X$, for all words $w \in F_r$ of length at most $2d^2$, where $d = \dim V$. Then by Lemma \ref{traces}, all these constants are contained in the field $\F_{q_1}$. But this cannot be the case, since by the previous paragraph together with Lemma \ref{traces} again, there must be a word $w$ of length at most $2d^2$ such that $\chi(w(x_2)) \in \F_{q_2}\setminus \F_{q_1}$. 

In the rest of the proof, $c_i\,(1\le i\le 10)$ denote positive absolute constants.
By the previous paragraph,  there exists a word $w$ of length at most $2d^2$ such that $\chi \circ w$ is not constant on $X$. Then by Lemma \ref{trap}, for each $q\in Q$, the number of elements $x \in X(q)$ such that $\F_p[\chi(w(x)] \ne \F_{q}$ is less than $c_1q^{(s-1)/2}$, where $s = \dim X$.

Now define $W = W_r(k)$ as in (\ref{wdef}). Since $q_2>M$ by assumption, we have $x_2 \in X\cap W$ (see Remark \ref{defM}), so $X\cap W \ne \emptyset$. By Lemma \ref{n large}(iii),  $W$ is defined over $\F_p$.

By the Lang-Weil theorem \cite{lw}, there is a positive constant $c_2$ such that for all sufficiently  large $q \in Q$, we have $|X(q)| > c_2q^{s}$. Also, $X$ is irreducible and $W$ is open dense by  Lemma \ref{n large}, so $X\setminus W$ is a proper closed subset of $X$. Hence another application of Lang-Weil gives $|X(q)\setminus W| < c_3q^{s-1}$ for some constant $c_3$. It follows that 
\[
|X(q)\cap W| > c_4q^{s}
\]
for all sufficiently large $q \in Q$, where $c_4$ is a positive constant. 
By the above, for all but at most $c_5q^{(s-1)/2}$ elements $x \in X(q)\cap W$, we have 
$\F_p[\chi(w(x)] \ne \F_{q}$, hence $\langle x \rangle = G(q)$. It follows that conclusion (i) of Theorem \ref{newmain1} holds. 

Now suppose $G(q)$ is a Suzuki or Ree group. Again we have $X\cap W \ne \emptyset$, and $X\setminus W$ is a proper closed subset of $X$.
 In this case  we cannot 
apply the Lang-Weil theorem as above, since it does not apply for the Suzuki-Ree type Frobenius morphism $F_q$.
Instead we use \cite[Lemma 2.2]{Lasht} (which is based on \cite{V}) to deduce that there are positive constants $c_i$ such that 
$c_6q^{s/2} < |X(q)| < c_7 q^{s/2}$ and $|X(q)\setminus W| < c_8q^{(s-1)/2}$. 

At this point we cannot continue as before, since Lemma \ref{trap} does not apply. Instead we use a counting method based on our hypothesis in Theorem \ref{newmain1}
that in the Suzuki-Ree case, the subvariety $X$ is invariant under $G$-conjugation. 
By the above, we have $|X(q)\cap W| > c_9q^{s/2}$ for some positive constant $c_9$. 
That is to say, for large $q$ there are at least $c_9q^{s/2}$ $r$-tuples in $X(q)$ that generate $G(q)$-conjugate of a subfield 
subgroup $G(q_0)$ of $G(q)$ for some $\FF_{q_0} \subseteq \FF_q$. Moreover, $s = \dim X > d =\dim G$ by Lemma \ref{xleg}. 

Let $q_0 = q^{\frac{1}{k}}$ with $k>1$, and let $G(q_0)$ be a fixed subfield subgroup of $G(q)$ (so $k$ is odd). 
Then the number of $r$-tuples in $X(q)$ that generate a $G(q)$-conjugate of $G(q_0)$ is at most 
\[
\left|\bigcup_{g\in G(q)}X(q_0)^g \right| \le |X(q_0)|\,|G(q):G(q_0)| \le c_{10}q_0^{s/2}q^{d/2}q_0^{-d/2} = c_{10}q^{\frac{d}{2}+\frac{s-d}{2k}} \le  c_{10}q^{\frac{s}{2}-\frac{1}{2}+\frac{1}{2k}}.
\]
The number of possible values of $k$ is at most $\log q$, so we conclude that the  total number of $r$-tuples in $X(q)$ that generate a proper subfield subgroup of $G(q)$ is at most 
$c_{10}q^{\frac{s}{2}-\frac{1}{3}}\log q$. It follows that $P_{G(q)}(X(q)) \to 1$ as $q\to \infty$, as required.

This completes the proof of Theorem \ref{newmain1}.

\subsection{Proof of Theorem \ref{newmain2}}\label{th2pf}

Let $G,X$ and $R= {\mathcal O}[\frac{1}{n}]$  be as in the statement of the theorem. 
 Define the subscheme $W_r'$ and its complement $W=W_r$ as in the preamble to Lemma \ref{n large}, and the constants $M=M_G$ and $K$  as in Remark \ref{defM} and Lemma \ref{goodmodules}. 

Suppose there exists a prime $p\nmid n$,  powers $q,q'$ of $p$ with $q'>M$, and $r$-tuples $x \in X(q)$, $x'\in X(q')$ such that $\langle x \rangle = G(q)$, $\langle x' \rangle = G(q')$ (if no such prime exists, then clearly neither (i) nor (ii) of Theorem \ref{newmain2} can hold). Then as above, there is a word $w$ of length at most $2d^2$ such that $\tr \circ w$ is non-constant on $X(k)$, where $k = \bar \FF_p$, $d = \dim V$ with $V$ as defined above, and traces are taken on this module (over $k$). As the subscheme $X$ is defined in characteristic zero over $R = {\mathcal O}[\frac{1}{n}]$, it follows that there are only finitely many primes $s$ for which $\tr \circ w$ can be constant on $X(\bar \FF_s)$. Let $S$ be this finite set of primes, together with the prime divisors of $n$. Then $p\not \in S$, and condition (ii) of Theorem \ref{newmain2} holds.

The subscheme $X$ is defined over $R = {\mathcal O}[\frac{1}{n}]$, 
and by its definition  the subscheme $W$ is defined over $\ZZ[\frac{1}{s}]$, where $s$ is the product of all the primes less than $K$. Hence for primes $l>K$ with $l \not \in S$, the number of components of the variety $X(\bar \FF_l)\setminus W(\bar \FF_l)$ does not depend on $l$, and so it is bounded by a constant depending only on the type of $G$. At this point the above proof of Theorem \ref{newmain1} goes through (using Corollary \ref{tra} instead of Lemma \ref{trap}) to show that condition (i) of Theorem \ref{newmain2} holds. 
This completes the proof of Theorem \ref{newmain2}.

Finally, we justify Remark (1) stated after Theorem \ref{newmain2} in the Introduction. Suppose that $\dim X > \dim G$, and there exists $x'\in X(q')$ such that $\langle x' \rangle = G(q')$ and $q'>M$. Then Corollary \ref{orderN} implies that 
$\tr \circ w$ is non-constant on $X(k)$ for some word $w$ of length at most $2d^2$, and now the proofs of Theorems \ref{newmain1} and \ref{newmain2} go through as above.

\subsection{Proof of Corollaries \ref{prob}--\ref{mainvary}}

$\;\;\;$

\vspace{2mm}
\noindent \textbf{Corollary \ref{prob}}:
 We shall apply Theorems \ref{newmain1} and \ref{newmain2},  taking $r=2$ and $X = G\times G$. 

In a fixed characteristic $p$, $X$ is irreducible, defined over $\F_p$, and also for any $q$, there is a pair in $X(q)$ generating $G(q)$ by \cite{St}. Hence Theorem \ref{newmain1} gives the conclusion (note that $X$ is of course $G$-invariant, so Theorem \ref{newmain1} applies for the Suzuki and Ree types). In particular this handles the Suzuki and Ree groups, as these are only defined in fixed characteristic (2 or 3). 

When the characteristic is allowed to vary, we apply Theorem \ref{newmain2}, noting that $X = G\times G$ is defined over $\ZZ$. This gives the conclusion of Corollary \ref{prob} in all but finitely many characteristics; and for the remaining characteristics, the conclusion follows as before using Theorem \ref{newmain1}.

\vspace{2mm}
\noindent \textbf{Corollary \ref{cgcor}}: We apply Theorem \ref{newmain1}, taking $X = C\times G$. For any $c \in C(q)$, there exists $d \in G(q)$ such that $\langle c,d\rangle = G(q)$, by \cite[Corollary, p.745]{GK}. Also, for $q \in Q$ sufficiently large, $C(q)$ is non-empty. Now Theorem \ref{newmain1} gives the conclusion. 

\vspace{2mm}
\noindent \textbf{Corollary \ref{mainp}}:
Now we prove Corollary \ref{mainp}. Let $k,G,C,D,Q$ be as in the statement. Then $X = C\times D$ is an irreducible subvariety of $G\times G$. Note that if $X(q)\ne \emptyset$, then $X$ is $F_q$-stable (hence $q \in Q$).

Assume condition (ii) of the statement -- so there exist $(c,d) \in C(q)\times D(q)$ such that $\langle c,d\rangle = G(q)$ and $q>M$. Then by Lemma \ref{scottlemma}, we have $\dim C+\dim D > \dim G$. Hence Theorem \ref{newmain1}, together with Remark (1) after Theorem \ref{newmain2}, implies that condition (i) of the Corollary holds.

\vspace{2mm}
\noindent \textbf{Corollary \ref{mainvary}}:
Again take $X = C\times D$. Regarding $G$ as a Chevalley group scheme, $X$ is defined over ${\mathcal O}[\frac{1}{rs}]$, where ${\mathcal O}$ is the ring of integers of a number field whose field of fractions is a cyclcotomic extension of $\QQ$. Arguing as above, Theorem \ref{newmain2} (and the ensuing Remark (1)) gives the result for all but finitely many characteristics, and Theorem \ref{newmain1} handles the remaining characteristics.

\subsection{Proof of Corollary \ref{rscor}}

Let $r,s$ be primes, not both 2. 
First of all, the set of limit points of $P_{r,s}(T)$ for finite 
simple groups $T$ which are alternating or classical of large rank 
(exceeding $f(r,s)$ where $f$ is a suitable function) is equal to $\{ 1 \}$: this follows from 
\cite{lishrs} for classical groups, and from \cite[Thm. 5.1 and Lemma 2.4]{lish23} for alternating groups.

Hence, to prove Corollary \ref{rscor} it suffices to study the limit
points of $P_{r,s}(T)$ where $T$ is a simple group of Lie type of 
bounded rank (at most $f(r,s)$). Since there are boundedly many 
types of such groups, it suffices to restrict to a single type, 
coming from an algebraic group $G$.

Let $R, S$ be the sets of elements of orders $r,s$ in $G$, respectively.
Since $R$ and $S$ are closed sets, their product
$R \times S$ has finitely many irreducible components, and 
we only need to consider the components of maximal dimension.
These have the form $C \times D$ where $C, D$ are conjugacy
classes. We now apply our previous results for these
classes: if the characteristic is bounded we may apply
Corollary \ref{mainp}; otherwise we may assume the characteristic
is larger than $r,s$ and so $C,D$ consist of semisimple elements,
so Corollary \ref{mainvary} is applicable. In any case we see that
the set of limit points for each component $C \times D$ is either $\{0\}$ 
or $\{1\}$. We conclude that set of the limit points of $P_{r,s}(G(q))$ as $q$ varies
is a finite set of rational numbers. The result follows.

\section{$(2,3)$-Generation}

In this section we prove Theorem~\ref{23excep}, namely that for simple 
exceptional groups of Lie type $G(q)$, excluding the Suzuki groups, we have 
$P_{2,3}(G(q)) \rightarrow 1$ for $q \rightarrow \infty$.
Here $G(q)$ is of one of the types
$^2G_2$, $G_2$, $^3D_4$, $^2F_4$, $F_4$, $E_6$, $^2E_6$, $E_7$ or $E_8$.

It is convenient to prove the result replacing $G(q)$ by its simply connected cover 
$G^{F_q}$ arising from a corresponding simple, simply connected algebraic group $G$ over $\bar \FF_q$
(the finite simple groups are quotients of these by their centres).

We wish to apply Corollaries~\ref{mainp} and~\ref{mainvary}. First we need to identify the largest classes of elements of orders 2 and 3 in 
$G$ (modulo $Z(G)$). This is done in the next result; in Table \ref{maxcl}, the labelling of the classes is taken from \cite{LSbk} for unipotent classes, and uses the centralizer of an element in the class for semisimple classes.

\begin{lemma}\label{largest}
Let $G$ be a simply connected simple algebraic group of exceptional type over an algebraically closed field of characteristic $p$. Then  $G$ has unique conjugacy classes of elements of orders 
$2$ and $3$ (modulo $Z(G)$) of maximal dimension, as listed in Table $\ref{maxcl}$. 
In all but one case, each class listed intersects $G(q)$ in one class; the exception is the class labelled $G_2(a_1)$ 
in $G_2\,(p=3)$, which gives two $G(q)$-classes.
\end{lemma}

\begin{table}[h] 
\caption{}
\label{maxcl}
\[
\begin{array}{|l|l|l|l|l|}
\hline
G & \hbox{largest invol.} &  \hbox{largest invol. } &  \hbox{largest order 3 } &  \hbox{largest order 3 }  \\
   & \hbox{class, }p=2 & \hbox{class, }p\ne 2 & \hbox{class, }p=3 & \hbox{class, }p\ne 3 \\ 
\hline 
E_8 & A_1^4, \dim 128 & D_8, \dim 128 & A_2^2A_1^2, \dim 168 &  A_8, \dim 168 \\
E_7 & A_1^4, \dim 70 & A_7, \dim 70 & A_2^2A_1, \dim 90 &  A_5A_2, \dim 70 \\
E_6 & A_1^3, \dim 40 & A_1A_5, \dim 40 & A_2^2A_1, \dim 54 &  A_2^3, \dim 54 \\
F_4 & A_1\tilde A_1, \dim 28 & A_1C_3, \dim 28 & \tilde A_2A_1, \dim 34 &  A_2\tilde A_2, \dim 34 \\
G_2 & \tilde A_1, \dim 8 & A_1\tilde A_1, \dim 8 & G_2(a_1), \dim 10 &  A_1T_1, \dim 10 \\
\hline
\end{array}
\]
\end{table}

\begin{proof}
For unipotent classes (involutions with $p=2$, order 3 elements with $p=3$), this follows from \cite[Chapter 22]{LSbk}. 
For semisimple elements of order 2 or 3, the classes correspond to orbits of the Weyl group on elements of order 2 or 3 in a maximal torus. These orbits and the corresponding centralizers are independent of the characteristic, and the classes and centralizers in the corresponding complex Lie groups are listed in \cite{CG, CW1, CW2}.
\end{proof} 

\vspace{4mm}
\noindent {\bf Proof of  Theorem~\ref{23excep}}

\vspace{2mm}
Suppose first that $G$ is of type $E_8,E_7,E_6$ or $F_4$. Let $C,D$ be classes of elements of orders 2,3 in $G$ of largest dimension. 
By the uniqueness of these classes given by Lemma \ref{largest}, there are positive absolute constants $c_1,c_2$ such that $|C(q)| > c_1\,i_2(G(q))$ and $|D(q)| > c_2\,i_3(G(q))$ for all $q$, where $i_r(G(q))$ denotes the number of elements of order $r$ in $G(q)$. Hence Theorem \ref{23excep} will follow from Corollaries~\ref{mainp} and~\ref{mainvary}, once we know that for sufficiently large $q$, $G(q)$ is generated by two elements, one  from each of the classes $C(q)$ and $D(q)$ (since this will ensure that condition (i) of Corollaries~\ref{mainp} and~\ref{mainvary} holds). But this has already been proved in \cite{LM}, with the exception of $G = E_8, p=3$. In the latter case, the $(2,3)$-generation of $E_8(q)$ was shown
in~\cite{LM} without using the largest class of elements of order~$3$. 
However, we have checked that the computations and arguments
from~\cite{LM} also work with the largest class of elements of order~$3$, completing the proof for this case.

It remains to handle the cases where $G(q)$ is of type $G_2$, $^2\!G_2$ or $^3\!D_4$. These groups were shown to be $(2,3)$-generated in \cite{malle1,malle2}.
With two exceptions, there are unique largest classes of elements of orders 2 and 3, and these are the classes used in \cite{malle1,malle2}, giving the result as above. 
The two exceptions are as follows. For $G(q) = G_2(q)$ (with $p=3$) there are two largest classes of
elements of order~$3$. Both contain an element which generates together with
an element of order~$2$ (there is only one class of these) 
the group $G_2(q)$, as shown in ~\cite{malle1}. Finally, for $G(q) = \,^3\!D_4(q)$ and $p \ne 3$ there are two classes of
elements of order $3$ and these have the same dimension. In~\cite{malle2}
the $(2,3)$-generation was shown using one of these classes. Using {\sf
CHEVIE}~\cite{chevie} one can check that similar computations and 
the same arguments work for the other class. 

This completes the proof of Theorem \ref{23excep}.

\end{document}